\documentclass{birkjour}
 \newtheorem{theorem}{Theorem}[section]
\newtheorem{lemma}[theorem]{Lemma}
\newtheorem{cor}[theorem]{Corollary}
\newtheorem{example}[theorem]{Example}
\theoremstyle{remark}
\newtheorem{remark}[theorem]{\bf{Remark}}
 \numberwithin{equation}{section}
 \usepackage{enumerate}
\usepackage{lipsum}
\usepackage{amsmath, amsthm, amscd, amsfonts, amssymb, graphicx, color}
\usepackage[bookmarksnumbered, colorlinks, plainpages]{hyperref}
\hypersetup{colorlinks=true,linkcolor=red, anchorcolor=green, citecolor=cyan, urlcolor=red, filecolor=magenta, pdftoolbar=true}

\begin{document}

%
%
%
%
%
%
%
%
%
\title [Upper bounds of numerical radius and $a$-numerical radius $\cdots$]{\Large{Upper bounds of numerical radius and $a$-numerical radius in $\mathcal{C}^*$-algebra setting using Orlicz functions}}

\author[Saikat Mahapatra]{Saikat Mahapatra}

\address{%
Department of Mathematics,\\ Jadavpur University, Kolkata 700032,\\ West Bengal, India.
}
\email{smpatra.lal2@gmail.com}

\thanks{The first author would like to thank UGC, Govt. of India for the financial support ( NTA
Ref. No. 211610170555) in the form of fellowship. The third author would like to thank the Science and Engineering Research Board, Govt. of India for the financial support (Grant SRG/2023/002420).}
\author{Riddhick Birbonshi }
\address{Department of Mathematics,\\ Jadavpur University, Kolkata 700032,\\ West Bengal, India.}
\email{riddhick.math@gmail.com}
\author{Arnab Patra}
\address{Department of Mathematics,\\  Indian Institute of Technology
Bhilai, 491001, \\Chhattisgarh, India.}
\email{arnabp@iitbhilai.ac.in}



\begin{abstract}
In this paper, several significant upper bounds for the numerical radius and $a$-numerical radius of an element in a $\mathcal{C}^*$-algebra are obtained using Orlicz functions. Many well-known results are obtained from our findings, depending on specific choices of Orlicz functions.
\end{abstract}

\maketitle
\section{Introduction and preliminaries}

Let $\mathcal{A} $ be a unital  $\mathcal{C}^*$-algebra with the unit denoted by  $\textbf{1}$ and $\mathcal{A}^*$ represents the dual space of $\mathcal{A}$. A linear functional $g\in \mathcal{A}^*$ is said to be positive if $g(x^*x)\ge 0$ for all $x\in \mathcal{A}$ and denoted by $g\ge 0$. By $\mathcal{S(A)}$ we denote the set of all states on $\mathcal{A}$, i.e., 
$\mathcal{S(A)}=\{g\in\mathcal{A}^* : g\ge 0 \hspace{0.15cm} and \hspace{0.15cm} g(\textbf{1})=1 \}$. The algebraic numerical range $V(x)$ and the algebraic numerical radius $v(x)$ of an element $x\in \mathcal{A} $ are defined respectively by \begin{equation*}
      V(x)=\{g(x):g\in \mathcal{S}\mathcal{(A)}\}, \hspace{0.15 cm}
   v(x)=\sup\{|z|:z\in V(x)\}.
\end{equation*} 
Let $\mathcal{B}(H)$ be the unital $\mathcal{C}^*$-algebra of all bounded linear operators defined on a complex separable Hilbert space $\big( H,\langle \cdot \rangle \big)$. In particular, if $\mathcal{A=B}(H)$   and $T\in \mathcal{B}(H)$  the spatial  numerical range and the spatial numerical radius of $T $ are denoted by $W(T)$ and $\omega(T)$ respectively. It is well known that $V(T)$ is the closure of $W(T)=\big\{ \big \langle Tx,x \big\rangle: x\in H,\|x\|=1\big\}$.

For any $x\in \mathcal{A}$ the fundamental inequality between numerical radius $v(x)$ and $\mathcal{C}^*$-norm  is as follows:
\begin{eqnarray}
    \frac{1}{2}\|x\|\le v(x) \le \|x\|. \label{in1}
\end{eqnarray}
A positive element $x\in \mathcal{A }$ is denoted by $x\ge 0$. To indicate that $x- y \ge 0$, we use $x\ge y$. Let  $\mathcal{A}^+$ denotes the set of all non-zero positive elements in $\mathcal{A}$. Now, for $a\in \mathcal{A}^{+}$ we define $\mathcal{S}_{a}\mathcal{(A)} =\{f\in \mathcal{A}^{*}:f\ge0,\hspace{0.07 cm}f(a)=1\}$ and it is nothing but 
$ \mathcal{S}_{a}\mathcal{(A)}= \big\{\frac{g}{g(a)}:g\in \mathcal{S(A)},\hspace{0.07 cm} g(a)\neq 0\big\}$.
 Throughout this paper, $a$ stands for a non-zero positive element of $\mathcal{A}$ unless otherwise specified. 
 Recently, $a$-numerical range and $a$-numerical radius of elements in unital $\mathcal{C}^*$-algebras were introduced and studied by Bourhim and Mabrouk in \cite{bourhim2021numerical}.
 The  $a$-numerical range and $a$-numerical radius of  an element $x\in \mathcal{A} $  are defined respectively, by 
\[V_{a}(x)=\{f(ax):f\in \mathcal{S}_{a}\mathcal{(A)}\},\hspace{0.15 cm}
   v_{a}(x)=\sup\{|z|:z\in V_{a}(x)\}. \] 
 In particular, when $a=\textbf{1}$, then $a$-numerical range and $a$-numerical radius of $x$ become algebraic numerical range and algebraic numerical radius of $x$ respectively, i.e., $V_a(x)=V(x)$ and 
$v_a(x)=v(x)$. These ideas were presented in  \cite{bourhim2021numerical}  as an extensions of the spatial $A$-numerical range and spatial $A$-numerical radius of a bounded linear operator $T\in \mathcal{B}(H)$ defined respectively by $W_A(T)=\big\{ \big \langle Tx,x \big\rangle_A: x\in H,\|x\|_A=1\big\}$ and $w_A(T)=\sup \{|z|: z\in W_A(T)\}$. Here $A$ be a positive bounded operator on a Hilbert space $(H,\langle \cdot\rangle)$ and $\|x\|_A^2=\langle x,x \rangle_A=\langle Ax,x\rangle$, $x\in H$.

 For an element $x\in \mathcal{A}$, let $\|x\|_{a}=\sup \{\sqrt{f(x^{*}ax}):f\in \mathcal{S}_{a}\mathcal{(A)}\}$. It is clear that  $\|x\|_{a}=0 $ if and only if $ax=0 $. Note that, it  may happen that $\|x\|_a=\infty$  for some $x\in \mathcal{A}$.
Let $\mathcal{A}^a$ denotes the set of all such elements $x\in \mathcal{A}$ such that $\|x\|_a< \infty.$ From \cite[Proposition 3.3]{bourhim2021numerical} it is known that $\|\cdot\|_a$ is a semi-norm on $\mathcal{A}^a$ and satisfies $\|xy\|_a\le \|x\|_a\|y\|_a$   for all $x,y\in \mathcal{A}^a$. Thus $\mathcal{A}^a$
 is a subalgebra of $\mathcal{A}$. 
For $x\in \mathcal{A}$, an element $x^{\#_a}\in \mathcal{A}$ is said to be an $a$-adjoint of $x$ if $ax^{\#_a}=x^*a.$ The existence and uniqueness of $a$-adjoint elements for $x \in \mathcal{A}$ are not assured. The collection of all elements in  $\mathcal{A}$  that have $a$-adjoints is denoted by $\mathcal{A}_a.$ That is 
\[\mathcal{A}_a=\{ x\in \mathcal{A}: \hspace{0.07 cm}there\hspace{0.1 cm} is\hspace{0.1 cm} x^{\#_a}\in \mathcal{A} \hspace{0.1 cm} such \hspace{0.1 cm} that \hspace{0.1 cm} ax^{\#_a}=x^*a \}. \]
Additionally, $\mathcal{A}_a$ is a subalgebra of $\mathcal{A}$ and 
$\mathcal{A}_a $ is a subset of  $\mathcal{A}^a$ (see \cite{bourhim2021numerical}).
If $x\in \mathcal{A}_a$ and $x^{\#_a}$ is an $a$-adjoint of $x$, then by \cite[Corollary 4.9]{bourhim2021numerical}\begin{equation*}\|x\|_a^2 =\|xx^{\#_a}\|_a=\|x^{\#_a}x\|_a=\|x^{\#_a}\|_a^2. \end{equation*}
An element $x \in \mathcal{A}$ is considered $a$-self adjoint if $ax$ is self-adjoint, i.e.,  $ax=x^*a$. It is noted that every 
 element $x$ in $\mathcal{A}_a$ can be expressed as $x=y+iz$, where $y$ and $z$ are $a$-self adjoint but this decomposition is not unique in general. In fact if $x^{\#_a}$ is an $a$-adjoint of $x$, then $x=\Re_a(x)+i\Im_a(x)$, where $ \Re_a(x)=\frac{x+x^{\#_a}}{2}$ and $\Im_a(x)=\frac{x-x^{\#_a}}{2i}$.  An element $x\in \mathcal{A} $ is said to be $a$-positive if $ax $ is positive, i.e., $ax\ge 0 $ and it is denoted by $x\ge_a 0$.\\
In \cite{bourhim2021numerical} , Bourhim and Mabrouk have shown that for any $x\in \mathcal{A}_a$, the inequality between $a$-numerical radius $v_a(x)$ and the semi-norm  $\|\cdot \|_a$ holds as below:
\begin{eqnarray}
    \frac{1}{2}\|x\|_a\le v_a(x) \le \|x\|_a. \label{in2}
\end{eqnarray}
  If $x $ is $a$-self-adjoint, the second part of the above inequality becomes equality and the first part of the above inequality becomes equality if $ax=0$.
 In \cite{mabrouk2023extension}, Mabrouk and Zamani further have given Some upper bounds of $a$-numerical radius, which are as follow:
\begin{eqnarray}
v_a^2(x)&\le& \frac{1}{2}\|xx^{\#_a}+x^{\#_a}x\|_a \label{in33}  \\
  v_a^2(x)&\le& \frac{1}{2}v_a(x^2)+\frac{1}{4}\|xx^{\#_a}+x^{\#_a}x\|_a \label{in3}
\end{eqnarray}
Clearly, inequality in (\ref{in33}) refines  the right side inequality  in (\ref{in2}) and the inequality in (\ref{in3}) is sharper than the inequality in (\ref{in33}) for the power inequality $v_a(x^2)\le v_a^2(x)$.
It is worth mentioning that several research articles exist in the literature that focus on the estimation of quantities $w(T)$, $w_A(T)$, $v(x)$, and $v_a(x)$. Here, we refer to the articles 
\cite{bhunia2020numerical, bhunia2021proper,bourhim2021numerical, dragomir2017generalizations, el2007numerical,feki2020note, feki2022some,kittaneh2005numerical, mabrouk2023extension,zamani2019characterization} and the books \cite{bhunia2022lectures,dragomir2013inequalities}.

In this article, by employing the concept of Orlicz functions, we obtain several upper bounds of $a$-numerical radius of an element in unital $\mathcal{C}^*$-algebras. In addition, few well-known results are established from our findings. Some examples are provided to supplement our results.

\section{Main Result}
Throughout this article, let $\mathcal{A}$ denotes a complex unital $\mathcal{C}^*$-algebra with unit $\textbf{1}$.
 We start this section with the following Lemmas which will be useful for our study.
 \begin{lemma} \cite{blackadar2006operator}
 Let $g$ be a non-zero positive linear functional on a $\mathcal{C}^*$-algebra $\mathcal{A}$. Then  $\big|g(x^*y)\big|^2\le g(x^*x)g(y^*y)$  for all  $x,y\in \mathcal{A}$.\label{gcauchy}\end{lemma}
As a consequence of Lemma \ref{gcauchy} we have the following result.
 \begin{lemma}
 Let $f\in \mathcal{S}_a\mathcal{(A)}$ and  $x,y\in \mathcal{A}$, then $\big|f(x^*ay)\big|^2\le f(x^*ax)f(y^*ay)$.   \label{cauchy}   
 \end{lemma}
  \begin{lemma} \cite{aujla2011refinements}
     Let $ \mathcal{A} $ be a unital $\mathcal{C}^*$
-algebra and $g\in \mathcal{S(A)}$. If $x$ is a self-adjoint element of $\mathcal{A}$ such that the spectrum of $x$ is contained in $[0,\infty)$, then for every continuous, convex function $\phi$ on $[0,\infty)$, the inequality $\phi\big(g(x)\big) \le g\big(\phi(x)\big)$ holds. \label{pp01}
 \end{lemma}
 \begin{cor} \label{pje}Let $\mathcal{A}$ be a unital  $C^*$-algebra.
     If  $x \in \mathcal{A}^+$, $g \in\mathcal{S}\mathcal{(A)}$ and $0<r\le 1$, then
     \[\big(g(x)\big)^r\ge g\big((x)^r\big).\]
 \end{cor}
 \begin{proof}
     Taking $\phi(t)=t^s$, $t\ge0$ and $s\ge 1$ in Lemma \ref{pp01}, we   get $\big(g(x)\big)^s\le g\big((x)^s\big)$.  Since $x \in \mathcal{A}^+$ so, $\big(g(x)\big)^r=\bigg(g\big((x^r)^{\frac{1}{r}}\big)\bigg)^r$. Since $\frac{1}{r}\ge 1$, we get $g\big((x^r)^\frac{1}{r}\big)\ge \big(g(x^r)\big)^{\frac{1}{r}}$. Hence the result follows.
 \end{proof}
If $\phi$ is a continuous, convex function on $[0, \infty) $ with $\phi(0)=0$, then using Lemma \ref{pp01} we get the following useful results.

\begin{lemma}\label{lj}
    Let $x $ be an a-self-adjoint element  in a unital $C^*$-algebra $\mathcal{A}$ such that the spectrum of $ax$ is contained in $[0,\infty)$ and $a\ge \textbf{1} $. If $\phi(t)$  is a continuous, convex function on $[0, \infty) $ with $\phi(0)=0$, then for all $f\in \mathcal{S}_a\mathcal{(A)}$
   \[ f\big(\phi(ax)\big)\ge \phi\big (f(ax)\big)   .\]
\end{lemma}
\begin{proof}
Let $f\in \mathcal{S}_a\mathcal{(A)}$. Then there exists $g\in \mathcal{S(A)}$ such that $f=\frac{g}{g(a)}$, where $g(a)\neq 0$. Since $a\ge \textbf{1}$, it implies $\frac{1}{g(a)}\le 1$. Again $\phi(0)=0$, it gives $\phi(\lambda t )\le \lambda \phi(t) $ for $0\le \lambda \le 1 $. Now, using Lemma \ref{pp01}, we get 
\[ f\big(\phi(ax)\big)=\frac{g\big(\phi(ax)\big)}{g(a)}\ge \frac{\phi\big(g(ax)\big)}{g(a)}\ge \phi\Bigg(\frac{g(ax)}{g(a)}\Bigg)=\phi\big(f(ax)\big).\]
\end{proof}
\begin{cor}\label{m02}
Let $\mathcal{A}$ be a unital  $\mathcal{C}^*$-algebra. If  $x$ be an $a$-self adjoint element in  $\mathcal{A}$, $f \in\mathcal{S}_a\mathcal{(A)}$ and $r\ge 1$, then
\[
\big(f(ax)\big)^{r}\le f\big((ax)^{r}\big).
\]
\end{cor}
\begin{proof} 
    Taking $\phi(t)=t^r$, $t\ge0$ and $r\ge 1$ in Lemma \ref{lj}, we  get the above inequality. 
\end{proof}
\begin{lemma}\label{lsplemma02}  Let $x ,y\in \mathcal{A}_a$ and $f\in\mathcal{S}_a\mathcal{(A)}$. Then
     \[ f(y^*x^*axy) \le \|x\|_a^2f(y^*ay).\]
    \end{lemma}
    \begin{proof}
        See the Proposition $3.3$ of \cite{bourhim2021numerical}.
    \end{proof}
Next result follows from \cite[Lemma 2.12]{bhanja2020generalized}.
\begin{lemma} \label{sbu}  Let $\big(H,\langle\cdot\rangle\big)$  be a semi inner product space. If $x,y,e \in \mathcal{H} $  with $\langle e,e\rangle=1$, then 
    \[
  \big |\langle x,e\rangle\langle e,y \rangle\big| \le \frac{1}{2}\Big(\sqrt{\langle x,x \rangle}\sqrt{\langle y,y \rangle}+\big|\langle x,y\rangle\big|\Big).\]
\end{lemma}
Let $\mathcal{A}$ be complex unital $\mathcal{C}^*$-algebra. Now, take $\mathcal{A}^n= \{(x_1,x_2,\cdots,x_n):x_i\in \mathcal{A}\}$. Let $(x_{1},x_{2},\cdots,x_n),(y_1,y_2,\cdots,y_n)\in\mathcal{A}^n $ and $\lambda \in \mathbb{C}$. Define \begin{eqnarray*}
 (x_{1},x_{2},\cdots,x_n)+(y_1,y_2,\cdots,y_n)&=&(x_{1}+y_1,x_2+y_2,\cdots,x_n+y_n)\\
\lambda(x_{1},x_{2},\cdots,x_n)& =&(\lambda x_{1}, \lambda x_{2},\cdots, \lambda x_n).   
\end{eqnarray*} With these operations $\mathcal{A}^n$ is a complex vector space.
The following two lemmas are generalizations of Lemma \ref{sbu} in $C^*$-algebra setting.
\begin{lemma}
    Let $\mathcal{A}$ be a unital $\mathcal{C}^*$-algebra and $P=(p_1,p_2, \cdots ,p_n)$ be a probability distribution  with $p_i\ge0$ for all $i=1,2,\cdots, n$ and $\sum_{i=1}^{n} p_i=1$. If $X=(x_1,x_2,\cdots,x_n)$, $Y=(y_1,y_2,\cdots,y_n)$ $\in \mathcal{A}^n$, and $f\in \mathcal{S}_a\mathcal{(A)}$, then
    \begin{eqnarray*}
        &&\Big| \sum_{i=1}^{n}p_if(ax_i)\Big|   \Big| \sum_{i=1}^{n}p_if(y_i^*a)\Big|\\
        &\le& \frac{  \Big| \sum_{i=1}^{n}p_if(y_i^*ax_i)\Big|+\sqrt{  \sum_{i=1}^{n}p_if(x_i^*ax_i) }\sqrt{ \sum_{i=1}^{n}p_if(y_i^*ay_i)}}{2}.
    \end{eqnarray*}    \label{cbuzano}
\end{lemma}
\begin{proof} Let $\mathcal{A}$ be a unital $\mathcal{C}^*-$algebra with unit $\textbf{1}$ and $X=(x_1,x_2,\cdots,x_n)$, $Y=(y_1,y_2,\cdots,y_n)$ , $I=(\textbf{1},\textbf{1},\cdots,\textbf{1})$ $\in \mathcal{A}^n$. Clearly, $X,Y,I\in \mathcal{A}^n$ and $\langle X,Y\rangle = \sum_{i=1}^{n}p_if(y_i^*ax_i)$ is a semi inner product on $\mathcal{A}^n $ with $\langle 
 I,I \rangle=1$.  Now, from Lemma  
 \ref{sbu} for the vectors $X,Y$ and $I$ we get the required inequality.
\end{proof}
As a consequence of Lemma \ref{cbuzano} we have the following result.
\begin{lemma}
     Let $x_i,y_i$ $\in \mathcal{A}_a$, $i=1,2,\cdots,n$. If  $ P=(p_1,p_2,\cdots,p_n )$ be a probability distribution  such that $p_i\ge0$ for all $ i=1,2,\cdots,n$ , $\sum_{i=1}^{n} p_i=1$ and $f\in \mathcal{S}_a\mathcal{(A)}$,  then
     \begin{eqnarray*}
          &&\Big| \sum_{i=1}^{n}p_if(ax_i)\Big|   \Big| \sum_{i=1}^{n}p_if(ay_i)\Big|\\
          &\le & \frac{1}{2}\sqrt{  \sum_{i=1}^{n}p_if(ax_i^{\#_a}x_i) }\sqrt{ \sum_{i=1}^{n}p_if(ay_iy_i^{\#_a})} + \frac{ 1}{2} \Big| \sum_{i=1}^{n}p_if(ay_ix_i)\Big|.
     \end{eqnarray*}
     \label{lal}
 \end{lemma}
 \begin{proof} Since $y_i\in \mathcal{A}_a$ so there is  $y_i^{\#_a}\in \mathcal{A} $ such that $ay_i^{\#_a}=y_i^*a$ for all $1\le i\le n$. Hence $Y^{\#_a}=(y_1^{\#_a},y_2^{\#_a},\cdots,y_n^{\#_a})\in \mathcal{A}^n$. Now, if we apply  Lemma \ref{cbuzano} for the elements $X=(x_1,x_2,\cdots,x_n)$ and $Y^{\#_a}=(y_1^{\#_a},y_2^{\#_a},\cdots,y_n^{\#_a})$, then we get the following inequality
\begin{eqnarray*}
        &&\Big| \sum_{i=1}^{n}p_if(ax_i)\Big|   \Big| \sum_{i=1}^{n}p_if\big((y_i^{\#_a})^*a\big)\Big|\\&\le & \frac{ 1}{2}\sqrt{  \sum_{i=1}^{n}p_if(x_i^*ax_i) }\sqrt{ \sum_{i=1}^{n}p_if\big((y_i^{\#_a})^*ay_i^{\#_a}\big)} +\frac{ 1}{2} \Big| \sum_{i=1}^{n}p_if\big((y_i^{\#_a})^*ax_i\big)\Big|.
        \end{eqnarray*}
       Again, we know that if  $f \in \mathcal{S}_a\mathcal{(A)}$ ,  then 
    $f(x^*)=\overline{f(x)}$. 
        Now, from the above inequality, we get
        \begin{eqnarray*}
          &&\Big| \sum_{i=1}^{n}p_if(ax_i)\Big|   \Big| \sum_{i=1}^{n}p_if(ay_i)\Big|\\
          &\le & \frac{1}{2}\sqrt{  \sum_{i=1}^{n}p_if(ax_i^{\#_a}x_i) }\sqrt{ \sum_{i=1}^{n}p_if(ay_iy_i^{\#_a})} 
          +  \frac{ 1}{2} \Big| \sum_{i=1}^{n}p_if(x_i^*ay_i^{\#_a})\Big|\\
          &\le & \frac{1}{2}\sqrt{  \sum_{i=1}^{n}p_if(ax_i^{\#_a}x_i) }\sqrt{ \sum_{i=1}^{n}p_if(ay_iy_i^{\#_a})}
          + \frac{1}{2}  \Big| \sum_{i=1}^{n}p_if(x_i^*y_i^*a)\Big|\\
          &\le & \frac{1}{2}\sqrt{  \sum_{i=1}^{n}p_if(ax_i^{\#_a}x_i) }\sqrt{ \sum_{i=1}^{n}p_if(ay_iy_i^{\#_a})}
          +  \frac{1}{2}  \Big| \sum_{i=1}^{n}p_if(ay_ix_i)\Big|.
          \end{eqnarray*}
         
 \end{proof}

Now we derive upper bounds of $a$-numerical radius of an element in $C^*$-algebra using Orlicz function.

A map $\phi:[0,\infty) \to [0,\infty) $ is said to be an Orlicz function if it is continuous, convex,  non-decreasing, $\phi(0)=0$  and $\phi(u)\to \infty$ as $u\to \infty$ (see \cite{lindenstrauss}). An Orlicz function is said to be non-degenerate if $\phi(u)>0$ when $u>0$ and it is said to be degenerate if $\phi(u)=0 $ for some $u>0$. Throughout this paper whenever the Orlicz function appears, we interpret it as non-degenerate.
 An Orlicz function can be expressed by  $ \phi(u)=\int_{0}^{u} p(t) dt$, where $p$ is a non decreasing function such that $p(0)=0$,  $p(t)>0$ for $t>0$, $\lim_{t\to \infty}p(t)=\infty$. When $\phi(u)$ is equivalent to the function $u$ then those restrictions on $p$ is excluded. Let $q$ be the right inverse of $p$  and it is defined as $q(s)=\sup\{t:p(t)\le s\}$, $s\ge
 0$. The Orlicz function $\psi$,  defined as  $ \psi(v)=\int_{0}^{v} q(s) ds$, is called the complementary Orlicz function of $\phi$. If $\phi(t)=\frac{t^P}{p}$, $t\ge 0$ and $p>1$ then the corresponding Orlicz function is $\psi(t)=\frac{t^q}{q}$, whwre $\frac{1}{p}+\frac{1}{q}=1$.  Recently, Manna and Majhi \cite{majhi2022improvements} derived several numerical radius inequalities using Orlicz functions.

\begin{lemma}  \label{young}\cite{lndenstrauss1996numerical}
    Let $\phi$, and $\psi $ be two complementary Orlicz functions. Then for $x,y\ge 0 $, $ xy\le \phi(x)+\psi(y).$
\end{lemma}
 The following lemma is used to prove the Theorem \ref{th1}. 

\begin{lemma} \label{lemma1}
    Let $x\in \mathcal{A}_a$ and $\phi$ be a Orlicz function. If $a\ge \textbf{1} $,  $f\in \mathcal{S}_a\mathcal{(A)}$ and   $0\le \alpha \le 1$, then
    \[\phi\Big(\big |f(ax)\big|^2\Big )\le   \alpha f\big ( \phi(ax^{\#_a}x)\big) + (1-\alpha) f\big( \phi(axx^{\#_a})\big) .\] 
\end{lemma}
 \begin{proof}
 Let $x\in \mathcal{A}_a$ and $f\in \mathcal{S}_a\mathcal{(A)}$. Then we have
     \begin{eqnarray*}
         \big|f(ax)\big|^2 &=& \alpha \big|f(ax)\big|^2+(1-\alpha)\big|f(ax)\big|^2\\
                  &=&  \alpha \big|f(ax)\big|^2+(1-\alpha)\big |\overline{f(ax)}\big|^2\\
                    &=&  \alpha \big|f(ax) \big|^2+(1-\alpha) \big|f(x^*a) \big|^2\\
                    &=& \alpha  \big |f(ax) \big|^2+(1-\alpha) \big|f(ax^{\#_a}) \big|^2\\
                    &\le& \alpha f(x^*ax) +(1-\alpha)f \big((x^{\#_a})^*ax^{\#_a} \big)\\
                      &=& \alpha f(ax^{\#_a}x)+(1-\alpha)f(axx^{\#_a}),
     \end{eqnarray*}
 where, the third equality comes from the fact that   $f(y^*)=\overline{f(y)}$ for all $y\in \mathcal{A}$ and the above inequality follows from Lemma \ref{cauchy}. 
     Now, using the  non-decreasing and convexity  property of $\phi$ and from Lemma \ref{lj}, we get
     \begin{eqnarray}
      \phi\Big( \big|f(ax) \big|^2\Big) &\le & \phi\Big( \alpha f(ax^{\#_a}x)+(1-\alpha)f(axx^{\#_a}) \Big)  \nonumber\\
      &\le & \alpha\hspace{0.1 cm} \phi  \big(f(ax^{\#_a}x) \big) + (1-\alpha) \phi \big(f(axx^{\#_a}) \big) \label{4.1}\\
      &\le & \alpha\hspace{0.1 cm}f \big ( \phi(ax^{\#_a}x) \big) + (1-\alpha) f \big( \phi(axx^{\#_a}) \big).  \nonumber
     \end{eqnarray}
 \end{proof}

\begin{theorem} \label{th1} Let $\mathcal{A} $ be a unital $\mathcal{C}^*$-algebra  and  $0\le \alpha \le 1$, then the following results hold:
\begin{enumerate}[\upshape (a)]
    \item   \label{ th1a} If $x\in \mathcal{A}_a $,  then 
       \[v_a^2(x) \le   \big\| \alpha x^{\#_a}x + (1-\alpha) xx^{\#_a} \big\|_a .\]
       \item \label{th1b} If $x\in \mathcal{A}$ and $\phi$ be an 
Orlicz function then 
         \[ \phi \big( v^2(x) \big)\le   \big\| \alpha \phi(x^{*}x) + (1-\alpha) \phi (xx^{*}) \big\| .
    \]
    \end{enumerate}
\end{theorem}
 \begin{proof}\begin{enumerate}[\upshape (a)]
     \item Considering $\phi(t)= t$, $t\ge 0$ in the inequality \eqref{4.1} one gets  
     \begin{eqnarray*}
    \big |f(ax)\big |^2 &\le & \alpha f(ax^{\#_a}x)+(1-\alpha)f(axx^{\#_a})\\
    \big |f(ax)\big |^2 &\le & f\big(a(\alpha x^{\#_a}x+(1-\alpha)xx^{\#_a})  \big).
     \end{eqnarray*}
     
     Now, taking supremum over $f\in 
  \mathcal{S}_a\mathcal{(A)}$ on both sides of the above inequality, we get 
  \[
  v_a^2(x)\le v_a\big(\alpha x^{\#_a}x+(1-\alpha)xx^{\#_a}\big)
  \]
  Again, we know that
  for any  $a$-self-adjoint element $y$,  $v_a(y)=\|y\|_a$ holds.
 Hence the result follows.
\item  Let $x\in \mathcal{A}$ and $f\in \mathcal{S(A)}$. Now,
     if we take $a=\textbf{1}$ in Lemma \ref{lemma1}, we get 
     \begin{eqnarray*}
           \phi( |f(x)|^2 )&\le &   \alpha\hspace{0.05 cm}f \big ( \phi(x^{*}x) \big) + (1-\alpha) f \big( \phi(xx^{*}) \big)\\
           &= &  f \big( \alpha\hspace{0.1 cm}  \phi(x^{*}x) + (1-\alpha) \phi(xx^{*}) \big).
     \end{eqnarray*}
     Taking supremum over $f\in 
  \mathcal{S}\mathcal{(A)}$ on both sides of the above inequality, we get
  \begin{eqnarray*}
          \phi( v^2(x))&\le&  v\big( \alpha \phi\big(x^{*}x\big) + (1-\alpha) \phi \big(xx^{*}\big)\big). 
    \end{eqnarray*}
    Since for a self-adjoint element $y\in \mathcal{A}$, $v(y)=\|y\|$, we get the desired inequality.
    \end{enumerate}
 \end{proof}
 Following corollary is an immediate consequence of Theorem \ref{th1}.
\begin{cor}\label{re01}
    Let $\mathcal{A}$ be a unital $\mathcal{C}^*$-algebra and $x\in \mathcal{A}$. If $0\le \alpha \le 1 $ and $r\ge 1$, then 
     \begin{eqnarray*}
    v^{2r}(x)\le  \big\|\alpha(x^*x)^{r}+(1-\alpha)(xx^*)^{r} \big\|.\hspace{0.5 cm} \end{eqnarray*}
\end{cor}
\begin{proof}
     Choose  $\phi(t)=t^r$, for $t\ge 0$ and $r\ge 1$ in 
 Theorem \ref{th1}(\ref{th1b}),
 we get the required result.
\end{proof}

\begin{remark} 
 Let $\mathcal{A} $ be a unital $\mathcal{C}^*$-algebra and  $0\le \alpha \le 1$, then we get the following known results from Theorem \ref{th1} and Corollary \ref{re01}.
    \begin{enumerate} [\upshape (i)]
        \item  If we take $\alpha=\frac{1}{2}$ in Theorem \ref{th1} (\ref{ th1a}), we get the right hand inequality of \cite[Corollary 3.5]{mabrouk2023extension}.
        \item  If we consider the $\mathcal{C}^*$-algebra $\mathcal{A}=\mathcal{B}(H)$, the Theorem \ref{th1} (\ref{ th1a}) becomes \cite[Theorem 1 ]{bhunia2021improvement}.
 \item  If  $\alpha=\frac{1}{2}$ and $\phi(t)=t$, $t\ge0$ in Theorem \ref{th1} (\ref{th1b}), we get \cite[Corollary 2.8]{zamani2019characterization}.
\item If we take $\mathcal{A}=\mathcal{B}(H)$, in the  Corollary \ref{re01} becomes  
  \cite[Theorem 2 ]{el2007numerical}.
\end{enumerate}
    
\end{remark}
Next result is a generalization of the inequality (\ref{in3}).
 
\begin{theorem} \label{ramm}
    Let $x\in \mathcal{A}_a$ and $\phi$ be an Orlicz function. Then 
    \begin{eqnarray*}
        \phi \big(v_a^2(x) \big) \le \frac{1}{2} \phi \big(v_a(x^2) \big)+\frac{1}{2}\phi\Bigg(\frac{\|xx^{\#_a}+x^{\#_a}x\|_a}{2}\Bigg).
    \end{eqnarray*}
\end{theorem}
\begin{proof}  Let $x\in \mathcal{A}_a$ and $f\in \mathcal{S}_a\mathcal{(A)}$. Now, taking $n=1$ in Lemma \ref{lal}, we get 
\begin{eqnarray*}
     \big|f(ax) \big|^2 \le \frac{1}{2} \big|f(ax^2) \big| +\frac{1}{2}\sqrt{f(ax^{\#_a}x)f(axx^{\#_a})}.
\end{eqnarray*}
Again, using the non-decreasing and convexity property of $\phi$ we obtain 
\begin{eqnarray*}
    \phi \Big( \big|f(ax) \big|^2 \Big)&\le& \frac{1}{2} \phi\Big( \big|f(ax^2) \big| \Big )+\frac{1}{2}\phi \Big(\sqrt{f(ax^{\#_a}x)f(axx^{\#_a})} \Big)\\
    &\le& \frac{1}{2} \phi \big(v_a(x^2) \big )+\frac{1}{2}\phi\bigg(\frac{f(ax^{\#_a}x)+f(axx^{\#_a})}{2}\bigg)\\
    &=& \frac{1}{2} \phi \big(v_a(x^2)  \big)+\frac{1}{2}\phi\bigg(\frac{f \big(a(x^{\#_a}x+xx^{\#_a}) \big)}{2}\bigg)\\
    &\le& \frac{1}{2} \phi \big(v_a(x^2) \big)+\frac{1}{2}\phi\bigg(\frac{\|xx^{\#_a}+x^{\#_a}x\|_a}{2}\bigg).
\end{eqnarray*}
Since $\phi$ is continuous, non-decreasing and the above inequality holds for all $f\in 
  \mathcal{S}_a\mathcal{(A)}$, so we get the desired result.

\end{proof}
An immediate consequence of Theorem \ref{ramm} is the following corollary.

\begin{cor}\label{re02}
      Let $x\in \mathcal{A}_a$. If $r\ge 1$, then 
       \begin{eqnarray*}
        v_a^{2r}(x)\le \frac{1}{2} v_a^r(x^2)+\frac{1}{2^{r+1}}\big\|xx^{\#_a}+x^{\#_a}x\big\|_a^r.
    \end{eqnarray*}
\end{cor}
\begin{proof}
    If we take $\phi(t)=t^r$, $t\ge 0$ and $r\ge 1$ in Theorem \ref{ramm}, we get the required result.
\end{proof}

\begin{remark}
    Let $\mathcal{A} $ be a unital $\mathcal{C}^*$-algebra. Then we get the following known results:
    \begin{enumerate} [\upshape (i)]
      \item If we take $\phi(t)=t$, $t\ge 0$ in Theorem \ref{ramm}, then we get  \cite[Theorem 3.3]{mabrouk2023extension}.      
    \item   If we take $\mathcal{A}=\mathcal{B}(H)$, the Corollary \ref{re02} becomes  
 \cite[Theorem 2 ]{bhunia2021improvement}.
    \end{enumerate}
\end{remark}

 \begin{lemma}\label{lemma2}
  Let $x\in \mathcal{A}_a$ and $\phi$ be an Orlicz function. If $a\ge \textbf{1}$, $f\in \mathcal{S}_a\mathcal{(A)}$ and   $0\le \alpha \le 1$, then

  \begin{enumerate}[\upshape(a)]
      \item 
   $ \phi\Big(\big|f(ax)\big|^2\Big)\le   \frac{\alpha}{2}\phi\Big(\big|f(ax^2)\big|\Big)+\frac{\alpha}{4}f\big(\phi(axx^{\#_a})\big)+(1-\frac{3\alpha}{4})f\big(\phi(ax^{\#_a}x)\big).$ \label{rms01}
   \item  $ \phi\Big(\big|f(ax)\big|^2\Big)\le   \frac{\alpha}{2}\phi\Big(\big|f(ax^2)\big|\Big)+\frac{\alpha}{4}f\big(\phi(ax^{\#_a}x)\big)+(1-\frac{3\alpha}{4})f\big(\phi(axx^{\#_a})\big).$ \label{rms02}
\end{enumerate}
    
 \end{lemma}
\begin{proof}  \hspace{2 cm}
\begin{enumerate}[\upshape(a)]
    \item 
Let $x\in \mathcal{A}_a$ and $f\in \mathcal{S}_a\mathcal{(A)}$. Then we have,

\begin{eqnarray}
     && \phi(|f(ax)|^2)\nonumber\\
        &=&\alpha\phi\big(|f(ax)|^2\big)+(1-\alpha)\phi\big(|f(ax)|^2\big) \label{kline1}\\ 
         &\le&\alpha\phi(|f(ax)|^2)+(1-\alpha)\phi(f(x^*ax)) \nonumber\\ 
              &\le&\alpha\hspace{0.1 cm}\phi\bigg(\frac{  | f(axx)|+\sqrt{  f(ax^{\#_a}x) }\sqrt{ f(axx^{\#_a})}}{2}\bigg)+(1-\alpha)\phi(f(ax^{\#_a}x)) \nonumber\\
              &\le & \frac{\alpha}{2}\phi(|f(ax^2)|)+\frac{\alpha}{2}\phi(\sqrt{  f(ax^{\#_a}x) f(axx^{\#_a})})+(1-\alpha)\phi(f(ax^{\#_a}x)) \nonumber\\ 
               &\le & \frac{\alpha}{2}\phi(|f(ax^2)|)+\frac{\alpha}{2}\phi\big(\frac{  f(ax^{\#_a}x)  +f(axx^{\#_a})}{2}\big)+(1-\alpha)\phi(f(ax^{\#_a}x)) \nonumber\\ 
                 &\le & \frac{\alpha}{2}\phi(|f(ax^2)|)+\frac{\alpha}{4}\phi(f(axx^{\#_a}))+(1-\frac{3\alpha}{4})\phi(f(ax^{\#_a}x)) \label{line2}\\
                  &\le  & \frac{\alpha}{2}\phi(|f(ax^2)|)+\frac{\alpha}{4}f(\phi(axx^{\#_a}))+(1-\frac{3\alpha}{4})f(\phi(ax^{\#_a}x))\hspace{0.1 cm}, \nonumber       
\end{eqnarray}

where the third and fifth inequality hold for the convexity property of $\phi$ and first inequality comes from Lemma \ref{cauchy}. The second inequality comes from  Lemma \ref{lal}. Again, from the inequality $2pq\le p^2+q^2$ for all $p,q\in \mathbb{R}^+$, we get the forth inequality. The last inequality follows from Lemma \ref{lj}.
\item Let $x\in \mathcal{A}_a$. Then for any $f\in \mathcal{S}_a\mathcal{(A)}$, $|f(ax)|=|f(ax^{\#_a})|$ holds. Now, if we replace  $|f(ax)|$ by $|f(ax^{\#_a})|$ in the inequality (\ref{kline1}) and  proceed similarly as in part  (\ref{rms01}) of this lemma, then we get  
\begin{eqnarray}
      &&\phi\Big(\big|f(ax)\big|^2\Big)\nonumber\\
                &\le &  \frac{\alpha}{2}\phi\Big(\big|f(ax^2)\big|\Big)+\frac{\alpha}{4}\phi\big(f(ax^{\#_a}x)\big)+(1-\frac{3\alpha}{4})\phi\big(f(axx^{\#_a})\big) \label{jkline2}\\
                  &\le  & \frac{\alpha}{2}\phi\Big(\big|f(ax^2)\big|\Big)+\frac{\alpha}{4}f\big(\phi(ax^{\#_a}x)\big)+(1-\frac{3\alpha}{4})f\big(\phi(axx^{\#_a})\big)\hspace{0.1 cm}, \nonumber            
\end{eqnarray}

\end{enumerate}
\end{proof}
\begin{theorem} \label{th2} Let $\mathcal{A} $ be a unital $\mathcal{C}^*$-algebra  and  $0\le \alpha \le 1$, then we get the following results
\begin{enumerate}[\upshape (a)]
    \item \label{th2a}
       If $x\in \mathcal{A}_a$, then 
   \begin{enumerate}[\upshape (i)]
        \item
         $ v_a^2(x)\le \frac{\alpha}{2}v_a(x^2)+\big\|\frac{\alpha}{4}xx^{\#_a}+(1-\frac{3\alpha}{4})x^{\#_a}x\big\|_a$ \label{th2ai}\\ 
           \item \label{th2aii} $v_a^2(x)\le \frac{\alpha}{2}v_a(x^2)+\big\|\frac{\alpha}{4}x^{\#_a}x+(1-\frac{3\alpha}{4})xx^{\#_a}\big\|_a$.\\
    \end{enumerate}
    
    \item \label{th2b} If $x\in \mathcal{A}$ and $\phi$ be an 
Orlicz function, then 
    \begin{enumerate}[\upshape (i)]
        \item \label{th2bi} $\phi( v^2(x))\le\frac{\alpha}{2}\phi(v(x^2))+\big\|\frac{\alpha}{4}\phi(xx^{*})+(1-\frac{3\alpha}{4})\phi(x^{*}x)\big\| $\\
          \item \label{th2bii} $ \phi( v^2(x))\le\frac{\alpha}{2}\phi(v(x^2))+\big\|\frac{\alpha}{4}\phi(x^{*}x)+(1-\frac{3\alpha}{4})\phi(xx^{*})\big\|$.
    \end{enumerate}
    \end{enumerate}
\end{theorem}
\begin{proof} \hspace{2 cm} \begin{enumerate}[\upshape (a)]
  \item Let $x \in \mathcal{A}_a$ and $f\in \mathcal{S}_a\mathcal{(A)}$.
  \begin{enumerate}[\upshape (i)]
      \item 
       Taking $\phi(t)=t$, $t\ge0$ in inequality \eqref{line2} and using the fact that $|f(ay)|\le \|y\|_a$ $\forall$ $y\in \mathcal{A}$, we get 
     \begin{eqnarray*}
           \big |f(ax)\big|^2 &\le &  \frac{\alpha}{2}\big|f(ax^2)\big|+\frac{\alpha}{4}f(axx^{\#_a})+(1-\frac{3\alpha}{4})f(ax^{\#_a}x)\\
           &\le &  \frac{\alpha}{2}v_a(x^2)+f\Big( a\big( \frac{\alpha}{4} xx^{\#_a}+(1-\frac{3\alpha}{4}) x^{\#_a}x\big)\Big)\\
           &\le & \frac{\alpha}{2}v_a(x^2)+\bigg\|\frac{\alpha}{4}xx^{\#_a}+(1-\frac{3\alpha}{4})x^{\#_a}x\bigg\|_a.
     \end{eqnarray*}
     Since the above inequality holds for all $f\in 
  \mathcal{S}_a\mathcal{(A)}$, so we get our required result. 
  \item  Taking $\phi(t)=t$, $t\ge0$ in inequality \eqref{jkline2}   and proceeding Similarly as in the first result of part (\ref{th2a}) of this theorem, we obtain the required inequality.
 \end{enumerate}
  
  \item Let $x \in \mathcal{A}$ and $f \in \mathcal{S(A)}$.
  \begin{enumerate}[\upshape (i)]
      \item    Taking $a=\textbf{1}$ in Lemma \ref{lemma2} (\ref{rms01}),  we get 
     \begin{eqnarray*}
           \phi\big( \big|f(x)\big|^2 \big)&\le &  \frac{\alpha}{2}\phi\big(\big|f(x^2)\big|\big)+\frac{\alpha}{4}f\big(\phi(xx^{*})\big)+(1-\frac{3\alpha}{4})f\big(\phi(x^{*}x)\big)\\
           &= &  \frac{\alpha}{2}\phi\big(\big|f(x^2)\big|\big)+f\Big(\frac{\alpha}{4}\phi(xx^{*})+(1-\frac{3\alpha}{4})\phi(x^{*}x)\Big)\\
            &\le &\frac{\alpha}{2}\phi\big(v(x^2)\big)+v\Big(\frac{\alpha}{4}\phi(xx^{*})+(1-\frac{3\alpha}{4})\phi(x^{*}x)\Big) \\
           &=& \frac{\alpha}{2}\phi\big(v(x^2)\big)+\bigg\|\frac{\alpha}{4}\phi(xx^{*})+(1-\frac{3\alpha}{4})\phi(x^{*}x)\bigg\|.
     \end{eqnarray*}
    Now, taking supremum over $f\in 
  \mathcal{S}\mathcal{(A)}$, we get
  \begin{eqnarray*}
            \phi\big( v^2(x)\big)\le\frac{\alpha}{2}\phi\big(v(x^2)\big)+\bigg\|\frac{\alpha}{4}\phi(xx^{*})+(1-\frac{3\alpha}{4})\phi(x^{*}x)\bigg\|.
    \end{eqnarray*}
     \item  Taking $a=\textbf{1}$ in Lemma \ref{lemma2} (\ref{rms02}) and proceeding Similarly as in the first result of part (\ref{th2b}) of this  theorem,  we obtain that required inequality.
    \end{enumerate}
  \end{enumerate}
 \end{proof}
\begin{cor}\label{ccc}
    Let $\mathcal{A} $ be a unital $\mathcal{C}^*$-algebra and  $x\in \mathcal{A}$. If $0\le \alpha \le 1$ and $r\ge 1$, then 
    \begin{enumerate}[\upshape (i)]
        \item  $ v^{2r}(x)\le\frac{\alpha}{2}v^r(x^2)+\big\|\frac{\alpha}{4}(xx^{*})^{r}+(1-\frac{3\alpha}{4})(x^*x)^{r}\big\| $\\
          \item $ v^{2r}(x)\le\frac{\alpha}{2}v^r(x^2)+\big\|\frac{\alpha}{4}(x^*x)^{r}+(1-\frac{3\alpha}{4})(xx^{*})^{r}\big\|$.
    \end{enumerate}
\end{cor}
\begin{proof}
    If we take $\phi(t)=t^r$, $t\ge 0$ and $r\ge 1$ in Theorem \ref{th2} (\ref{th2b}),  then we get the required results.
\end{proof}
 
\begin{remark} Let $\mathcal{A} $ be a unital $\mathcal{C}^*$-algebra and  $0\le \alpha \le 1$, then we get the following results:
    \begin{enumerate}[\upshape (i)]
        \item  If we we put $\alpha=1 $  in Theorem \ref{th2} (\ref{th2a}), it becomes \cite[Theorem 3.3 ]{mabrouk2023extension}.
        \item If we consider $\mathcal{A}=\mathcal{B}(H)$ in Theorem \ref{th2} (\ref{th2a}), we get \cite[Theorem 3]{bhunia2021improvement}. 
          \item   If we consider $\mathcal{A}=\mathcal{B}(H)$, the Corollary \ref{ccc} becomes  \cite[Theorem 2.11]{bhunia2021proper}. 
        
 \item  if we  take $\phi(t)=t$, $t\ge 0$ and $\alpha=1 $ in Theorem \ref{th2} (\ref{th2b}),    we get the upper bound of  \cite[Theorem 2.4 ]{zamani2019characterization}.
        
    \end{enumerate}
\end{remark}
 In order to obtain our next inequality that gives an upper bound for the $a$-numerical radius of sum of two elements of $\mathcal{A}_a$, we need the following lemma. 

\begin{lemma}\label{L01}
 Let $x,y\in \mathcal{A}_a $.  If $\psi_1$ be the complementary Orlicz function of $\psi_2$ and $ f\in \mathcal{S}_a\mathcal{(A)}$, then  
      \begin{eqnarray*}
         && \big|f\big(a(x+y)\big)\big|^2 \\  &\le&|f(ax)|^2+|f(ay)|^2+|f(ayx)|+\psi_1\Big(\sqrt{f(ax^{\#_a}x)}\Big)+\psi_2\Big(\sqrt{f(ayy^{\#_a})}\Big).
    \end{eqnarray*}
\end{lemma}
      \begin{proof} Let $x,y\in \mathcal{A}_a $  and $ f\in \mathcal{S}_a\mathcal{(A)}$. Then we have
           \begin{eqnarray*}
 && |f\big(a(x+y)\big)|^2\\
 &\le& |f(ax)|^2+|f(ay)|^2+2|f(ax)||f(ay)|\\
     &\le &|f(ax)|^2+|f(ay)|^2+|f(ayx)|+\sqrt{f(ax^{\#_a}x)}\sqrt{f(ayy^{\#_a})}\\
&\le&|f(ax)|^2+|f(ay)|^2+|f(ayx)|+\psi_1\Big(\sqrt{f(ax^{\#_a}x)}\Big)+\psi_2\Big(\sqrt{f(ayy^{\#_a})}\Big)\hspace{0.1 cm},
     \end{eqnarray*}
     where the second inequality comes from Lemma \ref{lal} and third inequality holds from Lemma \ref{young}.
      \end{proof}
      \begin{remark}\label{p001}
Taking $\psi_1(t)=\frac{t^2}{2}=\psi_2(t)$ in Lemma \ref{L01}, we get
         \begin{eqnarray*}
             && \big|f\big(a(x+y)\big)\big|^2\\
&\le&|f(ax)|^2+|f(ay)|^2+|f(ayx)|+\frac{f(ax^{\#_a}x)+f(ayy^{\#_a})}{2}\\
&\le& v_a^2(x)+v_a(y)^2+v_a(yx)+\frac{f\big(a(x^{\#_a}x+yy^{\#_a})\big)}{2}\\
&\le& v_a^2(x)+v_a(y)^2+v_a(yx)+\frac{\big\|x^{\#_a}x+yy^{\#_a}\big\|_a}{2}.
         \end{eqnarray*}
         
         Taking supremum over $f\in 
  \mathcal{S}_a\mathcal{(A)}$, we get  \begin{eqnarray}
         v_a^2(x+y)   &\le &  v_a^2(x)+v_a^2(y)+v_a(yx)+\frac{1}{2}\big\|x^{\#_a}x+yy^{\#_a}\big\|_a.\label{hhnn}
         \end{eqnarray}
Now, if we take $\mathcal{A}=\mathcal{B}(H)$, then the inequality (\ref{hhnn}) becomes that in   \cite[ Theorem 3.6]{bhunia2020numerical}.
\end{remark}  
      \begin{theorem} \label{sum2}
           Let $x,y\in \mathcal{A}_a $. Then for any $n\ge 1$
      \begin{eqnarray*}
         v_a^2(x+y)   &\le &  v_a^2(x)+v_a^2(y)+v_a(yx)+\frac{1}{n}\big\|x\|_a^n+\frac{n-1}{n}\|y\big\|_a^{\frac{n}{n-1}}.
         \end{eqnarray*}
      \end{theorem}
      \begin{proof}
    Let $x,y\in \mathcal{A}_a $.  Taking $\psi_1(t)=\frac{t^n}{n}$ and $\psi_2(t)=\frac{n-1}{n}t^{\frac{n}{n-1}}$, $t\ge 0$ and $n\ge 1$ in Lemma \ref{L01}, we get
         \begin{eqnarray*}
             && \big|f\big(a(x+y)\big)\big|^2\\
&\le&|f(ax)|^2+|f(ay)|^2+|f(ayx)|+\frac{1}{n}\big(f(ax^{\#_a}x)\big)^{\frac{n}{2}}\\
&&+\frac{n-1}{n}\big(f(ayy^{\#_a})\big)^{\frac{n}{2(n-1)}}\\
&\le& v_a^2(x)+v_a^2(y)+v_a(yx)+\frac{1}{n}\|x\|_a^{n}+\frac{n-1}{n}\|y^{\#_a}\|_a^{\frac{n}{n-1}}\\
&=&v_a^2(x)+v_a^2(y)+v_a(yx)+\frac{1}{n}\big\|x\|_a^n+\frac{n-1}{n}\|y\big\|_a^{\frac{n}{n-1}}.
         \end{eqnarray*}
         Now, taking supremum over $f\in 
  \mathcal{S}_a\mathcal{(A)}$, we get the desired inequality.
      \end{proof}
      Next, we consider an example and show that the bound obtained in Theorem $\ref{sum2}$ is better than that given in \cite[Theorem 3.6]{bhunia2020numerical}.
 \begin{example}
    Consider $\mathcal{A}=\mathcal{B}(H) $, $T=\begin{pmatrix}
    \frac{1}{2} & 0\\
    0 & \frac{1}{3}
\end{pmatrix} $, $ S=\begin{pmatrix}
    \frac{1}{4} & 0\\
    0 & \frac{1}{5}
\end{pmatrix} $, 
   and $A=\begin{pmatrix}
    1 & 0\\
    0 & 1
\end{pmatrix} $. 
Then Theorem 3.6 in \cite{bhunia2020numerical} gives $w^2(T+S)\le  0.59375$, whereas for $n=3$ Theorem \ref{sum2} gives $w^2(T+S)\le 0.5625$.

    \end{example}

    Next we obtain the following inequality for the sum of $n$ elements of $\mathcal{A}_a$. 
\begin{theorem}
   Let $\phi$ be an Orlicz function and $x_i\in \mathcal{A}_a$,  $i=1,2,\cdots,n$. If $n\ge 1$, then
  \begin{eqnarray*}
            && \phi\Big(v_a^2\big(\sum_{i=1}^{n}p_ix_i\big)\Big)\\
            &\le& \frac{1}{2} \phi\Big(v_a\big(\sum_{i=1}^{n}p_ix_i^2\big)\Big)+\frac{1}{2} \phi\Big(\big\|\sum_{i=1}^{n}p_ix_i^{\#_a}x_i\big\|_a^{\frac{1}{2}} \big\|\sum_{i=1}^{n}p_ix_ix_i^{\#_a}\big\|_a^{\frac{1}{2}}\Big),
         \end{eqnarray*} where $P=( p_1,p_2,\cdots,p_n) $ be a probability distribution  such that $p_i\ge0$ for all $i=1,2,\cdots,n$ and $\sum_{i=1}^{n} p_i=1$.   \label{sumpi}
\end{theorem}
\begin{proof} Let $x_i\in \mathcal{A}_a$ and $f \in \mathcal{S}_a\mathcal{(A)}$. Then we get
    \begin{eqnarray*}
            &&\phi\Big(\big|f\big(a\sum_{i=1}^{n}p_ix_i\big)\big|^2\Big)\\
            &=&\phi\Big(\big|f\big(\sum_{i=1}^{n}p_iax_i\big)\big|^2\Big)\end{eqnarray*}
            \begin{eqnarray*}
            &\le & \phi\bigg(\frac{  \big| \sum_{i=1}^{n}p_if(ax_i^2)\big|+\sqrt{  \sum_{i=1}^{n}p_if(ax_i^{\#_a}x_i) }\sqrt{ \sum_{i=1}^{n}p_if(ax_ix_i^{\#_a})}}{2}\bigg)\\
             &\le &  \frac{1}{2}\phi\Big( \big| \sum_{i=1}^{n}p_if(ax_i^2)\big|\Big)+\frac{1}{2}\phi\bigg(\sqrt{  \sum_{i=1}^{n}p_if(ax_i^{\#_a}x_i) }\sqrt{ \sum_{i=1}^{n}p_if(ax_ix_i^{\#_a})}\bigg)\\
              &= &  \frac{1}{2}\phi\Big( \big| f\big(a\sum_{i=1}^{n}p_ix_i^2\big)\big|\Big)+\frac{1}{2}\phi\bigg(\sqrt{  f\big(a\sum_{i=1}^{n}p_ix_i^{\#_a}x_i\big) }\sqrt{ f\big(a\sum_{i=1}^{n}p_ix_ix_i^{\#_a}\big)}\bigg)\\
             & \le&  \frac{1}{2}\phi\Big( \big| f\big(a\sum_{i=1}^{n}p_ix_i^2\big)\big|\Big)+\frac{1}{2} \phi\Big(\big\|\sum_{i=1}^{n}p_ix_i^{\#_a}x_i\big\|_a^{\frac{1}{2}} \big\|\sum_{i=1}^{n}p_ix_ix_i^{\#_a}\big\|_a^{\frac{1}{2}}\Big),
        \end{eqnarray*}
        where Lemma \ref{lal}  and non-decreasing property of $\phi$ give us  the first inequality and the second inequality holds for the convexity property of $\phi$. Now, using the fact  that $|f(ay)|\le \|y\|_a$ for all $y\in \mathcal{A}$, we obtain the last inequality.
        The desired inequality is obtained by taking supremum over $f\in \mathcal{S}_a\mathcal{(A)}$ on both sides of the inequality above.
       
\end{proof}
\begin{remark} Let $\mathcal{A}=\mathcal{B}(H)$ and $a$ be the identity operator. Now, if we take $\phi(t)=t$, $t\ge0$ in Theorem \ref{sumpi}, we get 
\cite[Theorem 2.6]{dragomir2017generalizations}.\end{remark}
       To prove the next theorem we need the following lemma.
\begin{lemma}\label{pie}
     Let $x,y,z\in \mathcal{A}_a$, $\|y\|_a\le 1$ and $\phi$ be an Orlicz function. If $a\ge \textbf{1}$ and $f\in\mathcal{S}_a\mathcal{(A)}$, then 
  \begin{eqnarray*} 
  \phi\Big( \big|f(axyz^{\#_a})\big|\Big)&\le& \frac{\|y\|_a}{2} \Big(f\big(\phi(axx^{\#_a})+\phi(azz^{\#_a})\big)\Big).
 \end{eqnarray*}

\end{lemma}
\begin{proof}
  Let $x,y,z\in \mathcal{A}_a$ and $f\in \mathcal{S}_a\mathcal{(A)}$. Then we have,
\begin{eqnarray}
   \big |f\big(axyz^{\#_a}\big)\big|&=&  |f\big((yz^{\#_a})^* x^{*}a\big)| \nonumber\\
    &=&  \big|f\big((yz^{\#_a})^*ax^{\#_a}\big)\big|\nonumber\\
    &\le& \big|f\big((x^{\#_a})^*ax^{\#_a}\big)\big|^{\frac{1}{2}}\big|f\big((yz^{\#_a})^*ayz^{\#_a}\big)\big|^{\frac{1}{2}}\nonumber\\
     &=& \big|f\big(axx^{\#_a}\big)\big|^{\frac{1}{2}}\big|f\big((z^{\#_a})^*y^*ayz^{\#_a}\big)\big|^{\frac{1}{2}}\nonumber\\
&\le & \|y\|_a\big|f\big(axx^{\#_a}\big)\big|^{\frac{1}{2}}\big|f\big((z^{\#_a})^*az^{\#_a}\big)\big|^{\frac{1}{2}}\nonumber\\
&= & \|y\|_a\big|f\big(axx^{\#_a}\big)\big|^{\frac{1}{2}}\big|f\big(azz^{\#_a}\big)\big|^{\frac{1}{2}}\nonumber\\
&\le & \|y\|_a\frac{f\big(axx^{\#_a}\big)+f\big(azz^{\#_a}\big)}{2}\hspace{0.1 cm},\nonumber\label{l4.7}
\end{eqnarray}
where the first equality comes from the fact that $f(y^*)=\overline{f(y)}$ for all $y\in \mathcal{A}$ and the first inequality holds from Lemma \ref{cauchy}. Now, the second inequality follows from Lemma \ref{lsplemma02}  and the last inequality holds because  $2pq\le p^2+q^2$ for any $p,q\in \mathbb{R}$. 
  Now, using the convexity, non-decreasing property of $\phi$ and $\phi(\lambda t)\le \lambda \phi(t) $ for $0\le\lambda\le 1 , t\ge 0$, we get
  \begin{eqnarray*}
      \phi\big(\big |f(axyz^{\#_a})\big|\big)&\le& \phi\bigg( \|y\|_a \frac{f\big(axx^{\#_a}\big)+f\big(azz^{\#_a}\big)}{2}\bigg)\\
    &\le& \frac{\|y\|_a}{2} \bigg(\phi\big(f(axx^{\#_a})\big)+\phi\big(f(azz^{\#_a})\big)\bigg).
  \end{eqnarray*}
Again, from Lemma \ref{lj}, we get 
\begin{eqnarray*}
     \phi\big(\big |f(axyz^{\#_a})\big|\big)&\le& \frac{\|y\|_a}{2} \bigg(f\big(\phi(axx^{\#_a})\big)+f\big(\phi(azz^{\#_a})\big)\bigg)\\
     &=& \frac{\|y\|_a}{2} \bigg(f\big(\phi(axx^{\#_a})+\phi(azz^{\#_a})\big)\bigg).
\end{eqnarray*}
\end{proof}
Now, we prove the following result.
\begin{theorem}\label{piet}
    Let $\mathcal{A} $ be a unital $\mathcal{C}^*$-algebra, then  the following results hold:
    \begin{enumerate}[\upshape (a)]
        \item\label{pieta} If $x,y,z\in \mathcal{A}_a$ and $\|y\|_a\le 1$, then 
        \begin{eqnarray*}
            v_a\big(xyz^{\#_a}\big) \le \frac{\|y\|_a}{2}\big\|xx^{\#_a}+zz^{\#_a}\big\|_a.
        \end{eqnarray*}
        \item \label{pietb}If $x,y,z\in \mathcal{A}$, $\|y\|\le 1$ and $\phi$ be an Orlicz function, then 
  \begin{eqnarray*} 
  \phi\Big( v\big(xyz^*\big)\Big)&\le& \frac{\|y\|}{2} \Big\|\phi(xx^*)+\phi(zz^*)\Big\|.
 \end{eqnarray*}
    \end{enumerate}
\end{theorem}
\begin{proof}
    \begin{enumerate}[\upshape (a)]
        \item  Taking $\phi(t)=t$, $t\ge 0$ in Lemma \ref{pie}, we get
        \begin{eqnarray*}
         \big |f\big(axyz^{\#_a}\big)\big|&\le&  \frac{\|y\|_a}{2}\Big(f\big(axx^{\#_a}\big)+f\big(azz^{\#_a}\big)\Big)\\
          &=&  \frac{\|y\|_a}{2}f\Big(a(xx^{\#_a}+zz^{\#_a})\Big)\\
           &\le&  \frac{\|y\|_a}{2}\big \|xx^{\#_a}+zz^{\#_a}\big \|_a,
        \end{eqnarray*} where the last inequality comes from the fact that $|f(ay)|\le \|y\|_a $ for all $y\in \mathcal{A}$. 
        Now, taking supremum over $f\in\mathcal{S}_a\mathcal{(A)}$,  we get the desired inequality.
        \item Putting $a=\textbf{1}$ in Lemma \ref{pie} and using the fact that $|f(y)|\le \|y\| $  for all $y\in \mathcal{A}$, we get 
       \begin{eqnarray*} 
  \phi\big( \big|f(xyz^{*})\big|\big)&\le& \frac{\|y\|}{2} \Big(f\big(\phi(xx^{*})+\phi(zz^{*})\big)\Big)\\
  &\le& \frac{\|y\|}{2} \Big\|\phi(xx^*)+\phi(zz^*)\Big\|.
 \end{eqnarray*}
 Since the above inequality holds for all $f\in\mathcal{S}\mathcal{(A)}$, so the desired result follows .

    \end{enumerate}
\end{proof}
\begin{cor} \label{ram}
       Let $\mathcal{A} $ be a unital $\mathcal{C}^*$-algebra and $x,y,z\in \mathcal{A}$. If $r\ge 1$ and $\|y\|\le 1$, then 
        \begin{eqnarray}
    v^r\big(xyz^*\big)  \le \frac{\|y\|}{2}\big \|(xx^*)^r+(zz^*)^r\big\| .  
    \end{eqnarray}
\end{cor}
\begin{proof}
    If we put $\phi(t)=t^r $, where $t\ge 0$ and $r\ge 1$  in Theorem \ref{piet}(\ref{pietb}), we get the required inequality.
\end{proof}
\begin{remark} 
\begin{enumerate}[\upshape (i)]
\item It is shown in \cite[Lemma $4.4$]{zamani2019numerical} that if $T,S,R\in \mathcal{B}_{A}(\mathcal{H})$, then \[ w_{A}(SRT^{\#_{A}})\le \frac{1}{2}\|TT^{\#_{A}}+SS^{\#_{A}}\|_A\|R\|_A.\]
Now, if we take $\mathcal{A}=\mathcal{B}(H)$ and put $x=S$, $y=R$ and $z=T$  in Theorem \ref{piet}(\ref{pieta}), we get the above upper bound of $w_{A}(SRT^{\#_{A}})$ for $\|R\|_A\le 1$.
     \item If we take $\mathcal{A}=\mathcal{B}(H)$ and $y$ is the identity operator $I$ in Corollary \ref{ram}, we get  
    \cite[Theorem 1]{dragomir2008power}.
    \end{enumerate}
\end{remark}
A Orlicz function  $\phi $  is said to be sub-multiplicative if for every $u,v\ge 0$, $\phi(uv)\le \phi(u)\phi(v)$ holds.
In the following result, we use the sub multiplicative property of $\phi$ to find an upper bound of the numerical radius.

 \begin{theorem}\label{ra} Let $x,y,z\in \mathcal{A}$ and  $\phi$ be a sub-multiplicative  Orlicz function. If $x,z$ are positive, $r\ge 2$ and $0\le \alpha\le 1$, then 
    \[\phi\Big(v^r\big(x^\alpha y z^{(1-\alpha)}\big)\Big) \le \phi \big(\|y\|^r\big) \Big \|  \alpha \phi (x^{r})+
            (1-\alpha)\phi ( z^{r})\Big\| . \]
    \end{theorem}
    
    \begin{proof}  Let $x,z$ be positive, $y\in \mathcal{A}$ and $f\in\mathcal{S(A)}$. Then we have, 
\begin{eqnarray*}
    \big|f\big(x^{\alpha}yz^{(1-\alpha)}\big)\big|^r &=&\big |f\big((yz^{(1-\alpha)})^*x^{\alpha}\big)\big|^r\\
&\le & \big(f(x^{2\alpha})\big)^{\frac{r}{2}}\big(f(z^{(1-\alpha)}y^*yz^{1-\alpha})\big)^{\frac{r}{2}}\\
&\le &\|y\|^r \big(f(x^{2\alpha})\big)^{\frac{r}{2}}\big(f(z^{2(1-\alpha)})\big)^{\frac{r}{2}}\\
&\le &\|y\|^r \big(f(x^{r\alpha})\big)\big(f(z^{r(1-\alpha)})\big)\\
&\le &\|y\|^r \big(f(x^{r})\big)^{\alpha}\big(f(z^{r})\big)^{(1-\alpha)}\\
&\le &\|y\|^r \Big( {\alpha}f(x^{r})+(1-\alpha)f(z^{r})\Big)\hspace{0.1 cm},
\end{eqnarray*}
Where the first inequality comes from Lemma \ref{gcauchy} and the second inequality comes form Lemma \ref{lsplemma02}. Third and forth inequalities follows from Corollary  \ref{m02} and Corollary \ref{pje} respectively. The last inequality holds for the fact that  $a^{\lambda}b^{1-\lambda}\le \lambda a+(1-\lambda)b$ where $a,b\ge 0$ and $0\le \lambda \le 1$.\\
Now, using non-decreasing, sub-multiplicative and convexity property of $\phi$, we get 
\begin{eqnarray*}
    \phi\Big( \big|f\big(x^{\alpha}yz^{(1-\alpha)}\big)\big|^r\Big) \le \phi\big(\|y\|^r\big) \Big( {\alpha}\phi\big(f(x^{r})\big)+(1-\alpha)\phi\big(f(z^{r})\big)\Big).
\end{eqnarray*}
Again, from Lemma \ref{pp01}, we get 
\begin{eqnarray*}
    \phi\Big( \big|f\big(x^{\alpha}yz^{(1-\alpha)}\big)\big|^r\Big)& \le &\phi\big(\|y\|^r\big) \Big(f\Big( {\alpha}\phi\big(x^{r}\big)+(1-\alpha)\phi\big(z^{r}\big)\Big)\Big)\\
    & \le &\phi\big(\|y\|^r\big) \Big\| {\alpha}\phi\big(x^{r}\big)+(1-\alpha)\phi\big(z^{r}\big)\Big\|.
\end{eqnarray*}
    Taking supremum over $f\in \mathcal{S(A)}$ of the above inequality, we get the desired result.
    \end{proof}
\begin{remark}   If we put $\phi(t)$=t, $t\ge 0$ and consider $\mathcal{A}=\mathcal{B}(H) $ in Theorem \ref{ra}, we get \cite[Theorem 3.3 ]{sattari2015some}.
    \end{remark}

  \begin{theorem}\label{dra}
     Let $\phi$ be a multiplicative  Orlicz function and $r,s\ge 1 $. 
     If $w,x,y,z \in \mathcal{A} $, then
         \begin{eqnarray*}
        &&\phi\bigg(\frac{v^2(x^*w+z^*y)}{2}\bigg) \\
        &\le&  \bigg\| \frac{\big(\phi(w^{*}w)\big)^r+ \big(\phi(y^{*}y)\big)^r}{2}\bigg\|^{\frac{1}{r}} \bigg\| \frac{\big(\big (\phi(x^{*}x)\big)^s+ \big(\phi(z^{*}z)\big)^s}{2}\bigg\|^{\frac{1}{s}}.\end{eqnarray*}
        
\end{theorem}
\begin{proof} Let $w,x,y,z \in \mathcal{A}$ and $f\in \mathcal{S}\mathcal{(A)}.$ Then  we get 
    \begin{eqnarray*}
        \Big|f\Big(\frac{x^{*}w+z^{*}y}{2}\Big)\Big|^2
        &=& \Big|\frac{f(x^{*}w)}{2}+\frac{f(z^{*}y)}{2}\Big|^2\\
         &\le& \Big( \Big|\frac{f(x^{*}w)}{2} \Big | + \Big|\frac{f(z^{*}y)}{2}\Big |\Big)^2 \\
           &\le& \Big(\frac{1}{2}  \sqrt{f(x^*x) } \sqrt{f(w^*w)}+\frac{1}{2}  \sqrt{f(z^*z) } \sqrt{f(y^*y)}\Big )^2\\
           &\le & \frac{1}{4}  \Big(f(w^*w)+f(y^*y)\Big)\Big(f(x^*x)+f(z^*z)\Big)\\
              &= &  \Big(\frac{1}{2}f(w^*w)+\frac{1}{2}f(y^*y)\Big)\Big(\frac{1}{2}f(x^*x)+\frac{1}{2}f(z^*z)\Big)\hspace{0.1 cm},
    \end{eqnarray*}
    where second inequality comes from Lemma \ref{gcauchy} and the third inequality comes from the fact that $(ab+cd)^2\le (a^2+c^2)(b^2+d^2)$ for all $a,b,c,d \ge 0$.

    Again,
    \begin{eqnarray*}
&&\phi\big(\big|f\big(\frac{x^{*}w+z^{*}y}{2}\big)\big|^2\big)\\
          &\le & \phi \big( \big(\frac{1}{2}f(w^*w)+\frac{1}{2}f(y^*y)\big)\big(\frac{1}{2}f(x^*x)+\frac{1}{2}f(z^*z)\big)\big)\\
           &= & \phi \big(\frac{1}{2}f(w^*w)+\frac{1}{2}f(y^*y)\big)\phi\big(\frac{1}{2}f(x^*x)+\frac{1}{2}f(z^*z)\big)\\
             &\le & \big(\frac{1}{2} \phi\big(f(w^{*}w)\big)+\frac{1}{2} \phi\big(f(y^{*}y)\big)\big) \big(\frac{1}{2} \phi \big(f(x^{*}x)\big)+\frac{1}{2} \phi\big (f(z^{*}z)\big)\big)\\
              &\le & \bigg( \frac{(\phi(f(w^{*}w)))^r+ (\phi(f(y^{*}y)))^r}{2}\bigg)^{\frac{1}{r}} \bigg( \frac{(\phi (f(x^{*}x)))^s+ (\phi (f(z^{*}z)))^s}{2}\bigg)^{\frac{1}{s}}\\
 &\le & \bigg( \frac{(f(\phi(w^{*}w)))^r+ (f(\phi(y^{*}y)))^r}{2}\bigg)^{\frac{1}{r}} \bigg( \frac{(f (\phi(x^{*}x)))^s+ (f (\phi(z^{*}z)))^s}{2}\bigg)^{\frac{1}{s}}\\
 &\le & \bigg( \frac{f((\phi(w^{*}w))^r)+ f((\phi(y^{*}y))^r)}{2}\bigg)^{\frac{1}{r}} \bigg( \frac{f( (\phi(x^{*}x))^s)+ f((\phi(z^{*}z))^s)}{2}\bigg)^{\frac{1}{s}}\end{eqnarray*}
           \begin{eqnarray*}
 &=& \bigg( \frac{f((\phi(w^{*}w))^r+ (\phi(y^{*}y))^r)}{2}\bigg)^{\frac{1}{r}} \bigg( \frac{f((\phi(x^{*}x))^s+ (\phi(z^{*}z))^s)}{2}\bigg)^{\frac{1}{s}}\\
&\le& \bigg\| \frac{(\phi(w^{*}w))^r+ (\phi(y^{*}y))^r}{2}\bigg\|^{\frac{1}{r}} \bigg\| \frac{( (\phi(x^{*}x))^s+(\phi(z^{*}z))^s}{2}\bigg\|^{\frac{1}{s}}.\\
\end{eqnarray*}
Where the first and the second inequalities  hold for non-decreasing and convexity property of $\phi$ respectively.
 The third inequality comes from the fact that   $(\frac{a+b}{2})^r\le (\frac{a^r}{2}+\frac{b^r}{2})$ for all $a,b\ge 0$ and $r\ge 1$. Next,  the forth inequality follows from Lemma-\ref{lj} and  the fifth inequality comes from Corollary \ref{m02}.
    Now, Taking supremum over $f\in\mathcal{S}\mathcal{(A)}$, we get the desired result.
\end{proof}

\begin{cor}Let $p,q \in \mathcal{A}$ and $r,s\ge 1 $ , then 
\[
    \phi\bigg(v^2\big(\frac{pq-qp}{2}\big)\bigg) \le \bigg\|\frac{\big(\phi(q^{*}q)\big)^r+\big(\phi(p^{*}p)\big)^r}{2}\bigg\|^{\frac{1}{r}}\bigg\|\frac{\big(\phi(qq^{*})\big)^s+\big(\phi(pp^{*})\big)^s}{2}\bigg\|^{\frac{1}{s}}.
\]\end{cor}
\begin{proof}
    If we take $x^{*}=p$, $w=q$, $z^{*}=q$ and $y=-p$ in Theorem \ref{dra}, then we get the required result.
\end{proof}
\begin{remark}
   If we take $\mathcal{A}=\mathcal{B(H)}$  and $\phi(t)=t$, $t\ge0$ in Theorem \ref{dra}, we  get of \cite[ Theorem 3 ]{dragomir2008power}.
\end{remark}

\begin{lemma}\cite{furuta2001invitation}\label{bbrr}
     Let $T=U|T|$ be  the polar decomposition of an operator $T$ on a Hilbert space $H$. Then for any positive number $\alpha$ the following holds
     \begin{eqnarray*}
         |T^*|^{\alpha}=U|T|^{\alpha}U^*.
     \end{eqnarray*} 
\end{lemma}
Let $x$ be an invertible element in a unital $\mathcal{C}^*$-algebra $\mathcal{A}$. Then there exist a unique unitary $u$ such that $x=u|x|$, where $|x|=(x^*x)^{\frac{1}{2}}$. Now, using the same technique as used in Lemma $\ref{bbrr}$,
it can be proven that if $x$ is an invertible element in a unital $\mathcal{C}^*$-algebra $\mathcal{A}$, then  $|x^*|^{\alpha}=u|x|^{\alpha}u^*$  for any positive number $\alpha$. 

\begin{lemma}\label{l426}
    Let $x,y,z\in \mathcal{A} $ and $y$ be an invertible element in $\mathcal{A}$. If $0\le \alpha \le 1$ and $ f \in\mathcal{S(A)}$, then
\[
\big|f(xyz)\big|^2\le f\big(x|y^*|^{2(1-\alpha)}x^*\big)f\big(z^*|y|^{2\alpha}z\big).    
\]
\end{lemma}
 \begin{proof} Let $y$ be an invertible in $\mathcal{A}$. Then by polar decomposition there exists a unique unitary element $u$ in $\mathcal{A}$ such that $y=u|y|$. Now,
     \begin{eqnarray*}
          \big|f(xyz)\big|^2 &= &  \big|f\big(xu|y|z\big)\big|^2 \\
          &= &  \big|f\big(xu|y|^{(1-\alpha)}|y|^{\alpha}z\big)\big|^2 \\
             &\le &  f\Big(xu|y|^{(1-\alpha)}|y|^{(1-\alpha)}u^*x^*\Big) f\Big(z^*|y|^{\alpha}|y|^{\alpha}z\Big)\\
              &=&  f\big(xu|y|^{2(1-\alpha)}u^*x^*\big)f\big(z^*|y|^{2\alpha} z\big)\\
              &=&  f\big(x|y^*|^{2(1-\alpha)}x^*\big)f\big(z^*|y|^{2\alpha} z\big)\hspace{0.1 cm},
     \end{eqnarray*}
     where the above inequality comes from the Lemma \ref{gcauchy} and the last equality comes from the fact 
     \begin{eqnarray}
         u|y|^{\alpha}u^*=|y^*|^{\alpha}.\label{furuta}
     \end{eqnarray}
 \end{proof}
\begin{lemma}\label{tla}
 Let $x,y,p,q,r,s\in \mathcal{A} $. Let $\psi_1$ be the complementary Orlicz function of $\psi_2$ and $\psi_3$ be the complementary Orlicz function of $\psi_4$. If $x,y$ are invertible,  $ f \in  \mathcal{S}
\mathcal{(A)}$ and $0\le \alpha \le 1$ , then
 \begin{eqnarray*}
    && \Big|f(pxq+rys)\Big|\\
          &\le & \psi_1 \Big(\Big( f\big(p|x^*|^{2(1-\alpha)}p^*\big)\Big)^{\frac{1}{2}}\Big)+\psi_2\Big(\Big(f\big(q^*|x|^{2\alpha}q\big)\Big)^{\frac{1}{2}}\Big) \\&& +  \hspace{0.08 cm}\psi_3\Big(\Big(f\big(r|y^*|^{2(1-\alpha)}r^*\big)\Big)^{\frac{1}{2}}\Big)+\psi_4\Big(\Big(f\big(s^*|y|^{2\alpha}s\big)\Big)^{\frac{1}{2}}\Big)\hspace{0.1 cm}. 
\end{eqnarray*} 
\end{lemma}
\begin{proof} Let $x,y,p,q,r,s\in \mathcal{A} $. Now,  $ f \in  \mathcal{S}
\mathcal{(A)}$   and $x,y$ are invertible. Then we have
\begin{eqnarray*}
    && |f(pxq+rys)|\\
       &\le&  |f(pxq)|+\big|f(rys)|\\
       &\le &   (f(p|x^*|^{2(1-\alpha)}p^*))^{\frac{1}{2}} (f(q^*|x|^{2\alpha}q))^{\frac{1}{2}}+(f(r|y^*|^{2(1-\alpha)}r^*))^{\frac{1}{2}}\big(f(s^*|y|^{2\alpha}s)\big)^{\frac{1}{2}}
       \\
          &\le & \psi_1 \big(( f(p|x^*|^{2(1-\alpha)}p^*))^{\frac{1}{2}}\big)+\psi_2\big((f(q^*|x|^{2\alpha}q))^{\frac{1}{2}}\big) \\&& +  \hspace{0.08 cm}\psi_3\big((f(r|y^*|^{2(1-\alpha)}r^*))^{\frac{1}{2}}\big)
+\psi_4\big((f(s^*|y|^{2\alpha}s))^{\frac{1}{2}}\big)\hspace{0.1 cm},
\end{eqnarray*} 
where the second inequality comes from Lemma \ref{l426} and last inequality follows from Lemma \ref{young}. 
\end{proof}
 \begin{theorem}\label{kit28} 
      Let  $x,y,p,q,r,s\in \mathcal{A} $. If $x,y$ are invertible and $0\le \alpha \le 1$, then for any $n\ge 2$,
\begin{eqnarray*}
v(pxq+rys)&\le& \frac{1}{n} \bigg\|\bigg(p|x^*|^{2(1-\alpha)}p^*\bigg)^{\frac{n}{2}}+\bigg(r|y^*|^{2(1-\alpha)}r^*\bigg)^{\frac{n}{2}}\bigg\|\\&+&\frac{n-1}{n} \bigg(\bigg\|q^*|x|^{2\alpha}q\bigg\|^{\frac{n}{2(n-1)}} + \bigg\| s^*|y|^{2\alpha}s\bigg\|^{\frac{n}{2(n-1)}} \bigg).
 \end{eqnarray*}
 \end{theorem}
\begin{proof}
   Putting $\psi_1(t)=\frac{t^n}{n}=\psi_3(t)$, $\psi_2(t)=\frac{n-1}{n}t^{\frac{n}{n-1}}=\psi_4(t)$ for $t\ge0$ and $n\ge 2$  in Lemma \ref{tla}, we get
   \begin{eqnarray*}
          &&  \Big|f\big((pxq+rys)\big)\Big|\\
 &\le &\frac{1}{n} \Bigg(\Big( f\big(p|x^*|^{2(1-\alpha)}p^*\big)\Big)^{\frac{n}{2}}+\Big(f\big(r|y^*|^{2(1-\alpha)}r^*\big)\Big)^{\frac{n}{2}}   \Bigg)\\&& +\frac{n-1}{n}\Bigg(  \Big(f\big(q^*|x|^{2\alpha}q\big)\Big)^{\frac{n}{2(n-1)}}
+\Big(f\big(s^*|y|^{2\alpha}s\big)\Big)^{\frac{n}{2(n-1)}}\Bigg)\end{eqnarray*}
\begin{eqnarray*}
 &\le &\frac{1}{n} \Bigg( f\Big(\big(p|x^*|^{2(1-\alpha)}p^*\big)^{\frac{n}{2}}\Big)+f\Big(\big(r|y^*|^{2(1-\alpha)}r^*\big)^{\frac{n}{2}}\Big)   \Bigg)\\&& +\frac{n-1}{n}\Bigg(  \Big\|q^*|x|^{2\alpha}q\Big\|^{\frac{n}{2(n-1)}}
+\Big\|s^*|y|^{2\alpha}s\Big\|^{\frac{n}{2(n-1)}}\Bigg)\end{eqnarray*}\begin{eqnarray*}
&=&\frac{1}{n} \Bigg( f\Big(\big(p|x^*|^{2(1-\alpha)}p^*\big)^{\frac{n}{2}}+\big(r|y^*|^{2(1-\alpha)}r^*\big)^{\frac{n}{2}}\Big)   \Bigg)\\&& +\frac{n-1}{n}\Bigg(  \Big\|q^*|x|^{2\alpha}q\Big\|^{\frac{n}{2(n-1)}}
+\Big\|s^*|y|^{2\alpha}s\Big\|^{\frac{n}{2(n-1)}}\Bigg)\\
&\le&\frac{1}{n} \Big\|\big(p|x^*|^{2(1-\alpha)}p^*\big)^{\frac{n}{2}}+\big(r|y^*|^{2(1-\alpha)}r^*\big)^{\frac{n}{2}}  \Big\|\\&& +\frac{n-1}{n}\Bigg(  \Big\|q^*|x|^{2\alpha}q\Big\|^{\frac{n}{2(n-1)}}
+\Big\|s^*|y|^{2\alpha}s\Big\|^{\frac{n}{2(n-1)}}\Bigg).
   \end{eqnarray*}
    Where Corollary \ref{m02} and the fact that $|f(y)|\le \|y\|$ for all $y\in \mathcal{A}$ are used to get the second inequality. Now, taking  supremum over $f\in 
  \mathcal{S}\mathcal{(A)}$, we get the desired result.
\end{proof}
In  the following example we show that the bound obtained in Theorem $\ref{kit28}$ is better than that given in \cite[Theorem 2]{kittaneh2005numerical}.
    \begin{example}
   Consider  $\mathcal{A}=\mathcal{B}(H) $, $A=\begin{pmatrix}
    \frac{1}{2} & 0\\
    0 & \frac{1}{3}
\end{pmatrix} $, $B=\begin{pmatrix}
    3 & 0\\
    0 & 4
\end{pmatrix} $, $ C=\begin{pmatrix}
    \frac{1}{2} & 0\\
    0 & \frac{1}{4}
\end{pmatrix} $, 
    $D=\begin{pmatrix}
    3 & 0\\
    0 & 5
\end{pmatrix} $ and $T=S=\begin{pmatrix}
    1 & 0\\
    0 & 1
\end{pmatrix} $.
Then Theorem 2 in \cite{kittaneh2005numerical} gives $w(AB+CD)\le  \frac{5929}{288}$, whereas for $n=3$ Theorem \ref{kit28} gives $w(AB+CD)\le  \frac{65+40\sqrt{5}}{12}$.
    \end{example}

    \section*{Declarations}

    \textbf{Conflict of interest} There is no competing interest.\\
 \textbf{Availability of data}    Not applicable.

\bibliographystyle{plain}
\bibliography{BiB}

\end{document}